\documentclass[12pt]{article}
\usepackage{graphicx,epsfig}
\usepackage{color}
\usepackage{amsmath,amssymb,mathrsfs,amsthm}
\usepackage{amsfonts}
\usepackage{multicol}
\usepackage{graphicx}
\usepackage{color}
\usepackage{pdfcolmk}
\usepackage[nice]{nicefrac}
\usepackage{amsthm}
\usepackage{latexsym}
\usepackage{dsfont}
\usepackage{setspace}
\usepackage{multirow}
\usepackage{enumitem}
\usepackage{soul}
\usepackage[a4paper, tmargin=1.5in, bmargin=1in, lmargin=1in, rmargin=1in, headheight=13.6pt]{geometry}
\usepackage[sort&compress,numbers]{natbib}
\usepackage[font=small,labelfont=bf,labelsep=space]{caption}

\usepackage{color}

\makeatletter

\date{}

\def\@citex[#1]#2{\if@filesw\immediate\write\@auxout{\string\citation{#2}}\fi
  \def\@citea{}\@cite{\@for\@citeb:=#2\do
    {\@citea\def\@citea{,\linebreak[0]\hskip0pt plus .2em}%
      \@ifundefined{b@\@citeb}%
    {{\bf ?}\@warning{Citation `\@citeb' on page \thepage\space undefined}}%
      \hbox{\csname b@\@citeb\endcsname}}}{#1}}

\newtheorem{theorem}{Theorem}[section]

\newtheorem{definition}{Definition}

\newtheorem{example}{Example}[section]

\newtheorem{rule-def}[theorem]{Rule}
\numberwithin{equation}{section}
\makeatletter

\begin{document}
\title{Solving coupled Lane-Emden equations by  Green's function and decomposition technique}
\author{Randhir Singh\thanks{Corresponding author. E-mail:~{randhir.math@gmail.com} } \\
\small Department of Mathematics, Birla Institute of Technology Mesra, Ranchi-835215, India.
}\maketitle{}
\begin{abstract}
\noindent  In this paper, the Green's function and decomposition technique is proposed  for solving the coupled Lane-Emden equations. This approach depends on constructing Green's function before establishing the recursive scheme for the series solution.
Unlike, standard Adomian decomposition method, the present method avoids solving a sequence of transcendental equations for the undetermined coefficients. Convergence and error estimation is  provided. Three  examples of coupled Lane-Emden equations are considered to demonstrate the accuracy of the current algorithm.
\end{abstract}
\section{Introduction}
This paper aims to extend the application of  the Adomian decomposition method with Green's function \cite{singh2013numerical,singh2014efficient} for solving the following  coupled Lane-Emden boundary value problems
 \begin{eqnarray}\label{sec1:eq1}
\left\{
  \begin{array}{ll}
\displaystyle y''_{1}(x)+\frac{\alpha_1}{x}y'_{1}(x)=f_1(x,y_1(x), y_2(x)),~~~~x\in (0,1) \vspace{0.15cm}\\
\displaystyle y''_{2}(x)+\frac{\alpha_2}{x}y'_{2}(x)=f_2(x,y_1(x), y_2(x)), \vspace{0.15cm}\\
y'_{1}(0)=0,~~y'_{2}(0)=0,\vspace{0.15cm}\\
a_{1} y_{1}(1)+b_{1} y_{1}'(1)=c_{1},~~a_{2} y_{2}(1)+b_{2} y_{2}'(1)=c_{2},
\end{array} \right.
\end{eqnarray}
where $a_1,a_2,b_1,b_2,c_1, c_2$ are real constants. In recent years, singular boundary value problems  for ordinary differential equations have been studied extensively \cite{m2003decomposition,inc2005different,mittal2008solution,wazwaz2011comparison,wazwaz2011variational,khuri2010novel,ebaid2011new,
wazwaz2013adomian,kumar2010modified,singh2013numerical,singh2014efficient,
singh2014adomian,singh2016effective,das2016algorithm,singh2017optimal,singh2018optimal,singh2018analytical,singh2019haar,singh2019modified,singh2020haar,singh2019analytic} and references therein.  However, we find only the following results on coupled Lane-Emden equations.

Recently, in \cite{muthukumar2014analytical,wazwaz2014study,duan2015oxygen} authors studied  \eqref{sec1:eq1} with boundary conditions $y'_{1}(0)=y'_{2}(0)=0,$ $y_{1}(1)=y_{2}(1)=1$ and $\alpha_1=\alpha_2=2$ that relates the concentration of the carbon substrate and the concentration of oxygen.  In \cite {flockerzi2011coupled,rach2014solving}, authors considered the  coupled Lane--Emden equations \eqref{sec1:eq1} with  boundary conditions $y'_{1}(0)=y'_{2}(0)=0$, $y_{1}(1)=1,~y_{2}(1)=2$ and   $\alpha_1=\alpha_2=2$ occurs in catalytic diffusion reactions.   In \cite{rach2014solving,duan2015oxygen}, the Adomian decomposition method was applied to obtain a convergent analytic approximate solution  of \eqref{sec1:eq1}  with  $\alpha_1=\alpha_2=2$. Later, in \cite{wazwaz2016variational}, the variational iteration method was applied to obtain approximations to solutions of \eqref{sec1:eq1} for shape factors $\alpha_1,\alpha_2=1,2,3$. In \cite{saadatmandi2018numerical}, the Sinc-collocation method was used   to obtain the solution of  \eqref{sec1:eq1}. In \cite{hao2015solving} authors used the reproducing kernel Hilbert space method for solving  to obtain the solution of  \eqref{sec1:eq1}.

\section{Adomian decomposition method}
Recently, many researchers \cite{adomian1983inversion,adomian1993,adomian1994,Adomian1994a,wazwaz2000approximate,wazwaz2001reliable,jang2008two} have applied  the Adomian decomposition method to deal with many different scientific models. According to the Adomian decomposition method we rewrite \eqref{sec1:eq1} in a operator form as
\begin{align}\label{sec1:e3}
\left\{
  \begin{array}{ll}
\displaystyle L_1y_{1}(x)=f_1(x,y_1(x),y_2(x)),~~~~x\in (0,1) \vspace{0.2cm}\\
\displaystyle L_2y_{2}(x)=f_2(x,y_1(x),y_2(x)),
\end{array} \right.
\end{align}
where $L_1=x^{-\alpha_1}\frac{d}{dx}\left[x^{\alpha_1}\frac{d}{dx}\right]$  and $L_2=x^{-\alpha_2}\frac{d}{dx}\left[x^{\alpha_2}\frac{d}{dx}\right]$ are differential operators and their  inverse integral operators are defined as
\begin{eqnarray}
\left\{
  \begin{array}{ll}
L_{1}^{-1}[\cdot]=\displaystyle\int\limits_{0}^{x} x^{-\alpha_1}\int\limits_{0}^{x} x^{\alpha_1}[\cdot] \ dx  dx,\vspace{0.1cm}\\
L_{2}^{-1}[\cdot]=\displaystyle \int\limits_{0}^{x} x^{-\alpha_2}\int\limits_{0}^{x} x^{\alpha_2} [\cdot] \ dx  dx.\\
\end{array} \right.
\end{eqnarray}
Operating   $L_{1}^{-1}[\cdot],~L_{2}^{-1}[\cdot]$ on \eqref{sec1:e3} and using $y_{1}'(0)=y_{2}'(0)=0$, we get
\begin{align}\label{sec1:e4}
\left\{
  \begin{array}{ll}
 y_1(x)=y_1(0)+L_{1}^{-1} [f_1(x,y_1(x), y_2(x))],\vspace{0.2cm}\\
 y_2(x)=y_2(0)+L_{2}^{-1} [f_2(x,y_1(x), y_2(x))].
 \end{array} \right.
\end{align}
According to the ADM, we decompose $y_i(x)$ and  $f_i(x,y_1(x), y_2(x))$ as
\begin{align}\label{sec1:e5}
\left\{
  \begin{array}{ll}
y_1(x)=\displaystyle\sum_{j=0}^{\infty}y_{1j}(x),~~~~f_1(x,y_1(x), y_2(x))=\sum_{j=0}^{\infty}A_{1j},\vspace{0.3cm}\\
y_2(x)=\displaystyle \sum_{j=0}^{\infty}y_{2j}(x),~~~~f_2(x,y_1(x), y_2(x))=\sum_{j=0}^{\infty}A_{2j},
 \end{array} \right.
\end{align}
where $A_{ij}$ are Adomian's polynomials \cite{adomian1983inversion} are given
\begin{eqnarray}\label{sec1:e6}
A_{in}=\frac{1}{n!}\frac{d^n}{d
\lambda^n}\left[f_{i}\left(x, \sum_{j=0}^{\infty} y_{1j}\ \lambda^j,~~ \sum_{j=0}^{\infty} y_{2j}\ \lambda^j \right)\right]_{\lambda=0},~~~i=1,2.
\end{eqnarray}
Substituting  \eqref{sec1:e5} into \eqref{sec1:e4}, we get
\begin{align}\label{sec1:e7}
\left\{
  \begin{array}{ll}
\displaystyle \sum_{j=0}^{\infty}y_{1j}(x)=y_{1}(0)+L_{1}^{-1}\bigg[\displaystyle\sum_{j=0}^{\infty}A_{1j}\bigg],\vspace{0.3cm}\\
\displaystyle \sum_{j=0}^{\infty}y_{2j}(x)=y_{2}(0)+L_{2}^{-1}\bigg[\displaystyle\sum_{j=0}^{\infty}A_{2j}\bigg].
 \end{array} \right.
\end{align}
 Upon comparing both  sides of  \eqref{sec1:e7}, we have
  \begin{eqnarray}\label{sec1:e8}
\left\{
  \begin{array}{ll}
\displaystyle y_{10}(x)=\delta_1,~~y_{20}(x)=\delta_2,~~~~~~~\vspace{0.3cm}\\
\displaystyle y_{1j}(x,\delta_1,\delta_2)= L_{1}^{-1}[A_{1, j-1}],\vspace{0.3cm}\\
\displaystyle y_{2j}(x,\delta_1,\delta_2)= L_{2}^{-1}[A_{2, j-1}],~~~j=1,2,3....
\end{array} \right.
\end{eqnarray}
where $y_{1}(0)=\delta_1,~y_{2}(0)=\delta_2$ are unknown constants to be determined. The $n$-term series solutions are given as
\begin{eqnarray}
\left\{
  \begin{array}{ll}
\displaystyle \phi_{1n}(x,\delta_1,\delta_2)=\displaystyle\sum_{j=0}^{n} y_{1j}(x,\delta_1,\delta_2),~~~~~~~\vspace{0.1cm}\\
\displaystyle \phi_{2n}(x,\delta_1,\delta_2)=\displaystyle \sum_{j=0}^{n} y_{2j}(x,\delta_1,\delta_2).
\end{array} \right.
\end{eqnarray}
 The unknown constants may  be obtained  by imposing  boundary condition at $x=1$ on $\phi_{in}(x,\delta_1,\delta_2)$, that leads to
 \begin{eqnarray}
\left\{
  \begin{array}{ll}
 a_1 \phi_{1 n}(1,\delta_1,\delta_2)+b_1 \phi'_{1 n}(1,\delta_1,\delta_2)-c_1=0,~~~~~~~\vspace{0.3cm}\\
  a_2 \phi_{2 n}(1,\delta_1,\delta_2)+b_2 \phi'_{2 n}(1,\delta_1,\delta_2)-c_2=0.
 \end{array} \right.
\end{eqnarray}
Solving above transcendental equations for $\delta_i$ require additional computational work, and $\delta_i$ may not be uniquely determined.

To avoid solving the above sequence of difficult transcendental equations, the Adomian decomposition method with Green's function was introduced in \cite{singh2013numerical,singh2014efficient}.  This technique relies on constructing Green's function before establishing the recursive scheme for the solution components. Unlike the standard Adomian decomposition method, this avoids solving a sequence of transcendental equations for the undetermined coefficients.

\section{Green's function and decomposition technique}
In this section, we extend the application of the Adomian decomposition method with Green's function \cite{singh2014efficient,singh2014adomian,singh2014approximate,singh2016numerical,singh2016efficient}, where we transformed the singular boundary value problem into the integral equation before establishing the recursive scheme for the approximate solution. To apply this technique to  coupled Lane-Emden boundary value problems  \eqref{sec1:eq1}, we  first consider the equivalent integral form of coupled Lane-Emden equation  \eqref{sec1:eq1} as
\begin{align}\label{sec2:e1}
\left\{
  \begin{array}{ll}
\displaystyle y_1(x)=\frac{c_1}{a_1}+\int\limits_{0}^{1} \ G_1(x,s) \ s^{\alpha_1}\ f_1\big(s,y_1(s),y_2(s)\big)ds,~~~~~~~\vspace{0.3cm}\\
\displaystyle y_2(x)=\frac{c_2}{a_2}+\int\limits_{0}^{1} \ G_2(x,s) \ s^{\alpha_2}\ f_2\big(s,y_1(s),y_2(s)\big)ds,
 \end{array} \right.
\end{align}
where  $G_i(x,s)$ are given by
\begin{align}\label{sec2:e2}
G_i(x,s)= \left   \{
  \begin{array}{ll}
  \ln s, & x\leq s,~~~~~~\alpha_i=1,~~~i=1,2, \vspace{.2cm} \\
   \ln x , & s\leq x,
\end{array}
\right.
\end{align}
and
\begin{align}\label{sec2:e3}
G_i(x,s)= \left   \{
  \begin{array}{ll}
   \displaystyle \frac{s^{1- \alpha_i}-1}{1- \alpha_i},  & x \leq s,  ~~~~\alpha_i>1,~~~~i=1,2, \vspace{.2cm} \\
    \displaystyle \frac{x^{1- \alpha_i}-1}{1-\alpha_i}, & s\leq x.
\end{array}
\right.
\end{align}
Substituting the series \eqref{sec1:e5} into  \eqref{sec2:e1}, we obtain
\begin{eqnarray}\label{sec2:eq7}
\left\{
  \begin{array}{ll}
 \displaystyle \sum_{j=0}^{\infty}y_{1j}(x)=\frac{c_1}{a_1}+\int\limits_{0}^{1} \ G_1(x,s) \ s^{\alpha_1}\ \sum_{j=0}^{\infty}A_{1j}\ ds,\\
 \displaystyle \sum_{j=0}^{\infty}y_{2j}(x)=\frac{c_2}{a_2}+\int\limits_{0}^{1} \ G_2(x,s) \ s^{\alpha_2}\ \sum_{j=0}^{\infty}A_{2j}\ ds.
\end{array}
\right.
\end{eqnarray}
Comparing components from both  sides of \eqref{sec2:eq7} we have the following  recursive scheme
\begin{eqnarray}\label{sec2:eq8}
\left.
  \begin{array}{ll}
\displaystyle y_{10}(x)=\frac{c_1}{a_1},~~~y_{20}(x)=\frac{c_2}{a_2},~~~\vspace{.1cm} \\
\displaystyle y_{1j}(x)=\int\limits_{0}^{1} \ G_1(x,s) \ s^{\alpha_1}\ A_{1,j-1}\ ds,~~~\vspace{.1cm} \\
\displaystyle y_{2j}(x)=\int\limits_{0}^{1} \ G_2(x,s) \ s^{\alpha_2}\ A_{2,j-1}\ ds.
  \end{array}
\right\}
\end{eqnarray}
Then, we obtain  the approximate series solutions as
\begin{eqnarray}\label{sec2:eq9}
\left.
  \begin{array}{ll}
\displaystyle \psi_{1n}(x)=\sum_{j=0}^{n} y_{1j}(x),~~~\vspace{.1cm}\\
\displaystyle \psi_{2n}(x)=\sum_{j=0}^{n} y_{2j}(x).
  \end{array}
\right\}
\end{eqnarray}
Unlike ADM or MADM, the proposed recursive schemes \eqref{sec2:eq9} do not require any computation of unknown constants.

\section{Convergence and error analysis}
 Let $E= \big(C [0,1], \|y\|\big)$ be a Banach space with norm
 \begin{align}\label{sec4:eq0}
\|y\|=\displaystyle\max\{\|y_1\|, \|y_2\|\}, ~~~~~~~y\in E,
\end{align}
 where, $\|y_1\|=\displaystyle\max_{ x\in I=[0,1]} |y_1(x)|$ and $\|y_2\|=\displaystyle\max_{ x\in I} |y_2(x)|$.

From \eqref{sec2:eq8} and \eqref{sec2:eq9}, we have
\begin{align}\label{sec4:eq1}
 \Psi_{n}&= \sum_{j=0}^{n} \textbf{y}_j(x)= \frac{\textbf{c}}{\textbf{a}}+\sum_{j=1}^{n}\bigg[\int\limits_{0}^{1} \textbf{G}(x,s)  \textbf{s}^{\alpha}  \textbf{A}_{j-1} ds\bigg]
= \frac{\textbf{c}}{\textbf{a}}+ \int\limits_{0}^{1} \textbf{G}(x,s)   \textbf{s}^{\alpha} \sum_{j=1}^{n} A_{j-1} ds,
\end{align}
where
 \begin{align*}
 \textbf{y}&=\bigg(\begin{array}{c} y_{1} \\ y_{2} \end{array}\bigg),~~ \textbf{y}_j=\bigg(\begin{array}{c} y_{1j} \\ y_{2j} \end{array}\bigg),\textbf{G}(x,s)=\bigg(\begin{array}{c} G_1(x,s) \\ G_2(x,s) \end{array}\bigg),~~\textbf{A}_j=\bigg(\begin{array}{c} A_{1j} \\ A_{2j} \end{array}\bigg)\\
 \textbf{f}&=\bigg(\begin{array}{c} f_1 \\ f_2 \end{array}\bigg),~~\textbf{s}^{\alpha}=\bigg(\begin{array}{c} s^{\alpha_1} \\ s^{\alpha_2}\end{array}\bigg),~~ \frac{\textbf{c}}{\textbf{a}}=\bigg(\begin{array}{c} \frac{c_1}{a_1} \\ \frac{c_2}{a_2} \end{array}\bigg),~~\Psi_{n}=\left(\begin{array}{c} \psi_{1n} \\ \psi_{2n}\end{array}\right).
\end{align*}

\begin{definition}
{\rm The function $f(x,y_1,y_2)$ satisfy Lipschitz  condition as
\begin{align}\label{sec4:eq2}
|f(x,y_1,y_2)-f(x,y_1^{*},y_2^{*})|\leq  \sum _{j=1}^{2} l_j |y_j-y_j^{*}|,~~\forall (x,y_1,y_2),~(x,y_1^{*},y_2^{*}) \in D,
\end{align}
where $D=\{{[0,1]\times R\times R } \}$, and   $l_1$, $l_2$ are  Lipschitz constants.}
\end{definition}
\begin{theorem}\label{sec4:eq4}
\rm{ Suppose that the nonlinear function $\textbf{f}(x,y_1,y_2)$ satisfy Lipschitz  condition \eqref{sec4:eq2}, then the series solution $\sum_{j=0}^{\infty} \textbf{y}_{j}$  defined by \eqref{sec2:eq9} is convergent whenever $\gamma< 1$.}
\end{theorem}
\begin{proof}
Define
$$\Psi_{0}=\textbf{x}_{0},~\Psi_{1}=\textbf{x}_{0}+\textbf{x}_{1},
\ldots,\Psi_{n}=\sum_{k=0}^{n} \textbf{y}_{k}.$$

For   $n>m$  and using  \eqref{sec4:eq1},  we have
\begin{align*}
\|\Psi_{n}-\Psi_{m}\|&= \max_{ x\in I}\bigg|  \int\limits_{0}^{1}\textbf{G}(x,s)  \textbf{s}^{\alpha}  \bigg(\sum_{j=1}^{n} \textbf{A}_{j-1}-\sum_{j=1}^{m} \textbf{A}_{j-1}\bigg)ds\bigg|.
 \end{align*}
Using $\sum_{j=0}^{n}  \textbf{A}_{j}\leq \textbf{f}(s,\psi_{1n},\psi_{2n})$ from (\cite{rach2008new}), we have
\begin{align*}
\|\Psi_{n}-\Psi_{m}\|&\leq  \max_{ x\in I}\bigg|\int\limits_{0}^{1}\textbf{G}(x,s)  \textbf{s}^{\alpha}  \big[\textbf{f}\big(s,\psi_{1,n-1},\psi_{2,n-1})-\textbf{f}(s,\psi_{1,m-1},\psi_{2,m-1}\big)\big]ds\bigg|.
\end{align*}
Applying  the Lipschitz  condition, we get
\begin{align*}
\|\Psi_{n}-\Psi_{m}\| &\leq   \max_{ x\in I}\bigg|\int\limits_{0}^{1}\textbf{G}(x,s)  \textbf{s}^{\alpha}  ds\bigg| \sum _{i=1}^{2} l_i \max_{ x\in I} |\psi_{i,n-1}-\psi_{i,m-1}|\\
&\leq  \max_{ x\in I}\bigg|\int\limits_{0}^{1}\textbf{G}(x,s)  \textbf{s}^{\alpha}  ds\bigg| 2l  \max \bigg\{  \|\psi_{1,n-1}-\psi_{1,m-1}\|,   \|\psi_{2,n-1}-\phi_{2,m-1}\| \bigg\}\\
&\leq   2 \textbf{m} l  \|\Psi_{n-1}-\Psi_{m-1}\|=\delta \|\Psi_{n-1}-\Phi_{m-1}\|,
\end{align*}
where
\begin{align}\label{sec4:eq5}
\nonumber &\textbf{m}=\max \bigg\{\max\big|\int\limits_{0}^{1}G_1(x,s)  s^{\alpha_1}  ds\big|, \max\big|\int\limits_{0}^{1}G_2(x,s)  s^{\alpha_2}  ds\big|\bigg\},\\
&l=\max \{l_1,l_2\},~~\gamma=  2\ \textbf{m} l.
\end{align}
Thus, we have
\begin{align}\label{sec4:eq6}
\|\Psi_{n}-\Psi_{m}\|\leq \gamma \|\Psi_{n-1}-\Psi_{m-1}\|.
\end{align}
By taking   $n=m+1$ in \eqref{sec4:eq6}, we see that
\begin{align*}
\|\Psi_{m+1}-\Psi_{m}\|&\leq \gamma \|\Psi_{m}-\Psi_{m-1}\| \leq \gamma^2 \|\Psi_{m-1}-\Psi_{m-2}\|
\leq \ldots \leq \gamma^{m} \|\Psi_{1}-\Psi_{0}\|.
%\|\phi_{in+1}-\phi_{in}\|&=\|x_{in+1}\|\leq \gamma_i \|x_{in}\|\leq  \gamma_i^2 \|x_{in-1}\|\leq \ldots \leq \gamma_i^{n-k_0+1} \|x_{ik_{0}}\|,~~~ i=1,2.
\end{align*}
For all $n,m\in \mathbb{N}$, with  $n>m$, consider
\begin{align*}
&\|\Psi_{n}-\Psi_{m}\|=\|(\Psi_{n}-\Psi_{n-1})+(\Psi_{n-1}-\Psi_{n-1})+\cdots+(\Psi_{m+1}-\Psi_{m})\|\\
&\leq \|\Psi_n-\Psi_{n-1}\|+\|\Psi_{n-1}-\Psi_{n-2}\|+\cdots+\|\Psi_{m+1}-\Psi_{m}\|\\
&\leq   [\gamma^{n-1} + \gamma^{n-2} +\cdots+ \gamma^{m}] \ \|\Psi_{1}-\Psi_{0}\|\\
&=  \gamma^{m}[ 1+\gamma  + \gamma^{2} +\cdots+\gamma^{n-m-1} ] \ \|\textbf{y}_{1}\|\\
&=\gamma^{m}\left( \frac{1-\gamma^{n-m}}{1-\gamma}\right)\|\textbf{y}_{1}\|.
\end{align*}
It  follows that as  $\gamma<1$,
\begin{eqnarray}\label{sec4:eq7}
\|\Psi_{n}-\Psi_{m}\|\leq&\displaystyle\frac{\gamma^{m}}{1-\gamma}\|\textbf{y}_{1}\|.
\end{eqnarray}
Letting  $n, m \to \infty$, we obtain $\lim_{n,m\rightarrow \infty} \|\Psi_{n}-\Psi_{m}\|=0.$ Hence,  $\{\Psi_{n}\}$ is a Cauchy sequences in the Banach space  $E$.
\end{proof}

\begin{theorem}\label{sec4:eq8}
{\rm If the approximate solution $\Psi_{n}$ converges to $\textbf{y}(x)$,
 then the maximum absolute truncated error is estimated
 \begin{align}\label{sec4:eq9}
\|\textbf{y}(x)-\Psi_{m}\|\leq\frac{\gamma^{m} \ \textbf{m} }{1-\gamma}  \max_{x\in I}\big|\textbf{f}(x,y_{10},y_{20}) \big|.
\end{align}}
\end{theorem}
\begin{proof}
From  \eqref{sec4:eq7}, we have
\begin{align*}
\|\Psi_{n}-\Psi_{m}\|\leq&\displaystyle\frac{\gamma^{m}}{1-\gamma}\|\textbf{y}_{1}\|.
\end{align*}
Since  $\Psi_{n}\rightarrow \textbf{x}(t)$ as $n\rightarrow \infty$, and the above inequality reduces to
\begin{align}\label{sec4:eq10}
\|\textbf{y}(x)-\Psi_{m}\|\leq\frac{\gamma^{m}}{1-\gamma}   \|\textbf{y}_{1}\|.
\end{align}
From \eqref{sec2:eq8}, we have $\textbf{y}_{1}=\int\limits_{0}^{1}\textbf{G}(x,s)  \textbf{s}^{\alpha} \ \textbf{A}_{0} ds$,
and we find
\begin{align}\label{sec4:eq11}
 \|\textbf{y}_{1}\|&= \max_{x\in I}\bigg|\int\limits_{0}^{1}\textbf{G}(x,s)  \textbf{s}^{\alpha}  \textbf{A}_{0} ds \bigg|\leq |\textbf{m} \max_{x\in I}\big|\textbf{f}(x,y_{10},y_{20}) \big|.
\end{align}
Combining \eqref{sec4:eq10} and \eqref{sec4:eq11},  we obtain error estimate as
\begin{align}
\|\textbf{y}(t)-\Psi_{m}\|\leq\frac{\gamma^{m} \     \textbf{m} }{1-\gamma}  \max_{x\in I}\big|\textbf{f}(x,y_{10},y_{20}) \big|.
\end{align}
which completes the proof.
\end{proof}

\section{Numerical Results}
In this section, we consider three coupled Lane-Emden type boundary value problems to examine the accuracy of the present method. Since the exact solution of the problems is not known, we examine the accuracy and applicability of the present method by the absolute residual error
\begin{align}
\left\{
  \begin{array}{ll}
  r_{1n}(x):= \displaystyle  \bigg|\psi''_{1n}(x)+\frac{\alpha_1}{x}\psi'_{1n}(x)+f_1\bigg(x,\psi_{1n}(x),\psi_{2n}(x)\bigg)\bigg|,~~~~~x\in(0,1),\vspace{0.20cm}\\
 r_{2n}(x):= \displaystyle  \bigg|\psi''_{2n}(x)+\frac{\alpha_2}{x}\psi'_{2n}(x)+f_2\bigg(x,\psi_{1n}(x),\psi_{2n}(x)\bigg)\bigg|,\\
\end{array} \right.
\end{align}
where $ [r_{1n}, r_{2n}]^{T}$ is the absolute residual error and  $ [\psi_{1n}, \psi_{2n}]^{T}$  is the present approximate  solution. The maximum residual errors are defined as
\begin{align}
\left\{
  \begin{array}{ll}
 maxr_{1n}:= \displaystyle \max_{ x\in [0,1]}r_{1n}(x),\vspace{0.35cm}\\
maxr_{2n}:= \displaystyle \max_{ x\in [0,1]}r_{2n}(x).
\end{array} \right.
\end{align}

\subsection{The catalytic diffusion reactions problem \cite{flockerzi2011coupled,rach2014solving}}
\begin{example}\label{sec5:eq1}
{\rm Consider the coupled Lane-Emden equation  occurs in catalytic diffusion reactions \cite{flockerzi2011coupled,rach2014solving} as
\begin{align}
\left\{
  \begin{array}{ll}
 \displaystyle  y''_{1}(x)+\frac{2}{x}y'_{1}(x)=-k_{1} \ y_{1}^{2}(x)-k_{2} \ y_1(x)  y_2(x),~~~~~x\in(0,1),\vspace{0.1cm}\\
 \displaystyle  y''_{2}(x)+\frac{2}{x}y'_{2}(x)=-k_{3}\ y_{1}^{2}(x)-k_{4}\ y_1(x) y_2(x),\vspace{0.1cm}\\
 y'_{1}(0)=0,~~~y'_{2}(0)=0,~~~y_{1}(1)=1,~~y_{2}(1)=2,
\end{array} \right.
\end{align}
where the parameters $ k_{1}, k_{2}, k_{3}$ and $k_{4}$ are the actual chemical reactions. Here $a_1=a_2=1$, $b_1=b_2=0$, $c_1=1$ and $ c_2=2$.}
\end{example}

In Tables \ref{tab1} and \ref{tab3}, we list the numerical results of the approximate solution and the absolute error obtained by the proposed method of  Example \ref{sec5:eq1} for ($k_{1}=1,k_{2}=2/5,k_{3}=1/2, k_{4}=1$) and ($k_{1}=k_{2}=k_{3}=k_{4}=1/2$), respectively. We also compare the numerical results of the maximum residual error $[maxr_{1n}, maxr_{2n}]$ obtained by the present method and the results obtained by the modified ADM \cite{rach2014solving} in Table \ref{tab2}. In Table \ref{tab4}, we list the numerical results of the maximum residual error.

\subsubsection{When $k_{1}=1,k_{2}=2/5,k_{3}=1/2, k_{4}=1$}
By applying the proposed scheme \eqref{sec2:eq8} with the initial guesses $y_{10}(x)=1, y_{20}(x)=2$, we obtain the 5-terms series solutions as
\begin{align*}
\left\{
  \begin{array}{ll}
\psi_{15}(x)=0.776218+0.199501 x^2+0.018823 x^4+0.005706 x^6-0.0003741 x^8+0.000125 x^{10},\\
\psi_{25}(x)=1.68423+0.283069 x^2+0.0258518 x^4+0.007133 x^6-0.0004418 x^8+0.000152 x^{10}.
\end{array} \right.
\end{align*}
\begin{table}[htbp]
\caption{Numerical results of the approximate solution $ [\psi_{1n}(x), \psi_{2n}(x)]$ and the absolute error  $ [r_{1n}(x), r_{2n}(x)]$ when $k_{1}=1,k_{2}=2/5,k_{3}=1/2, k_{4}=1$ of  Example \ref{sec5:eq1}}
\centering
\setlength{\tabcolsep}{0.045in}
\begin{tabular}{l|cc| cc|cc|cc}
\hline
\cline{1-9}
$x$ & $\psi_{15}(x)$ & $\psi_{25}(x)$  & $r_{15}(x)$ & $r_{25}(x)$ &$\psi_{1,10}(x)$ &$\psi_{2,10}(x)$ &$r_{1,10}(x)$ &$r_{2,10}(x)$ \\
\cline{1-9}
0.1	&	0.7782151	&	1.6870682	&	7.00E-2 &	8.79E-2 	&	0.7836523	&	1.6938487	&	5.51E-3 	&	6.84E-3 	\\
0.2	&	0.7842287	&	1.6955995	&	6.55E-2	&	8.23E-2 	&	0.7893632	&	1.7020042	&	5.10E-3 	&	6.34E-3 	\\
0.3	&	0.7943299	&	1.7099257	&	5.85E-2	&	7.36E-2 	&	0.7989874	&	1.7157379	&	4.48E-3 	&	5.57E-3 	\\
0.4	&	0.8086434	&	1.7302167	&	4.97E-2	&	6.26E-2 	&	0.8126872	&	1.7352661	&	3.71E-3 	&	4.63E-3 	\\
0.5	&	0.8273577	&	1.7567278  &	3.98E-2	&	5.02E-2 	&	0.8306972	&	1.7609008	&	2.89E-3  	&	3.62E-3 	\\
0.6	&	0.8507387	&	1.7898165	&	2.97E-2	&	3.76E-2 	&	0.8533324	&	1.7930602	&	2.10E-3  	&	2.64E-3 	\\
0.7	&	0.8791464	&	1.8299638	&	2.02E-2	&	2.56E-2 	&	0.8809992	&	1.8322829	&	1.39E-3	   &	1.75E-3 	\\
0.8	&	0.9130553	&	1.8778003	&	1.18E-2	&	1.51E-2 	&	0.9142115	&	1.8792486	&	7.90E-4	   &	1.01E-3 	\\
0.9	&	0.9530791	&	1.9341363	&	5.09E-3	&	6.53E-3	    &	0.9536120	&	1.9348042	&	3.30E-4	    &	4.29E-4	\\
\hline
\end{tabular}
\label{tab1}
\end{table}

\begin{table}[htbp]
\caption{Comparison of the numerical results of the maximum residual error $[maxr_{1n}, maxr_{2n}]$  when $k_{1}=1,k_{2}=2/5,k_{3}=1/2, k_{4}=1$ of  Example \ref{sec5:eq1}}
\centering
\addtolength{\tabcolsep}{20pt}
\begin{tabular}{c|cc|cc}
\hline
\multirow{2}{*} & \multicolumn{2}{c|}{The present method} & \multicolumn{2}{c}{Modified ADM \cite{rach2014solving}}  \\
 \cline{2-5}
{$n$}      & $maxr_{1n}$          & $maxr_{2n}$       & $maxr_{1n}$      & $maxr_{2n}$    \\ 
\hline
2	&	8.72E-2	&	1.12E-1	&	4.13E-1	&	5.67E-1	\\
3	&	3.95E-2	&	5.05E-2	&	2.36E-1	&	3.09E-1	\\
4	&	2.00E-2	&	2.53E-2	&	6.43E-2	&	8.52E-2	\\
5	&	1.07E-2	&	1.35E-2	&	4.79E-2	&	6.09E-2	\\
6	&	6.01E-3	&	7.58E-3	&	2.09E-2	&	2.51E-2	\\
7	&	3.47E-3	&	4.37E-3	&	1.09E-2	&	1.36E-2	\\
8	&	2.06E-3	&	2.59E-3	&	6.21E-3	&	7.32E-3	\\
9	&	1.24E-3	&	1.56E-4	&	3.35E-3	&	3.28E-3	\\
10	&	7.60E-4	&	9.53E-4	&	1.79E-3	&	2.09E-3	\\
11	&	5.91E-4	&	5.91E-4	&	9.61E-4	&	1.11E-3	\\
\hline
\end{tabular}
\label{tab2}
\end{table}

\subsubsection{When $k_{1}=k_{2}=k_{3}=k_{4}=1/2$}
On applying the proposed scheme \eqref{sec2:eq8} with the initial guesses  $y_{10}(x)=1, y_{20}(x)=2$, we obtain the 5-terms series solutions as
\begin{align*}
\left\{
  \begin{array}{ll}
&\psi_{15}(x)=0.80364+0.17704 x^2+0.017049 x^4+0.00223 x^6-0.00001 x^8+0.0000316 x^{10},\\
&\psi_{25}(x)=1.80365+0.17704 x^2+0.017049 x^4+0.00223 x^6-0.00001 x^8+0.0000316 x^{10}.
\end{array} \right.
\end{align*}

\begin{table}[htbp]
\caption{Numerical results of the approximate solution $ [\psi_{1n}(x), \psi_{2n}(x)]$ and the absolute error  $ [r_{1n}(x), r_{2n}(x)]$ when $k_{1}=k_{2}=k_{3}=k_{4}=1/2$ of  Example \ref{sec5:eq1}}
\centering
\setlength{\tabcolsep}{0.04in}
\begin{tabular}{l|cc |cc|cc|cc}
\hline
\cline{1-9}
$x$ & $\psi_{15}(x)$ & $\psi_{25}(x)$  & $r_{15}(x)$ & $r_{25}(x)$ &$\psi_{1,10}(x)$ &$\psi_{2,10}(x)$ &$r_{1,10}(x)$ &$r_{2,10}(x)$ \\
\cline{1-9}
0.1	&	0.8054213	&	1.8054213	&	1.42E-2	&	1.42E-2	&	0.8065106	&	1.8065106	&	2.63E-4	&	2.63E-4	\\
0.2	&	0.8107583	&	1.8107583	&	1.33E-2	&	1.33E-2	&	0.8117871	&	1.8117871	&	2.43E-4	&	2.43E-4	\\
0.3	&	0.8197227	&	1.8197227	&	1.18E-2	&	1.18E-2	&	0.8206563	&	1.8206563	&	2.13E-4	&	2.13E-4	\\
0.4	&	0.8324215	&	1.8324215	&	1.00E-2 &	1.00E-2	&	0.8332325	&	1.8332325	&	1.76E-4	&	1.76E-4	\\
0.5	&	0.8490101	&	1.8490101	&	8.05E-3	&	8.05E-3	&	0.8496803	&	1.8496803	&	1.37E-4	&	1.37E-4	\\
0.6	&	0.8696982	&	1.8696982	&	5.98E-3	&	5.98E-3	&	0.8702190	&	1.8702190	&	9.91E-5	&	9.93E-5	\\
0.7	&	0.8947568	&	1.8947568	&	4.05E-3	&	4.05E-3	&	0.8951290	&	1.8951290	&	6.53E-5	&	6.53E-5	\\
0.8	&	0.9245277	&	1.9245277	&	2.36E-3	&	2.36E-3	&	0.9247601	&	1.9247601	&	3.71E-5	&	3.71E-5	\\
0.9	&	0.9594352	&	1.9594352	&	1.00E-3	&	1.01E-3	&	0.9595423	&	1.9595423	&	1.55E-5	&	1.55E-5	\\
\hline
\end{tabular}
\label{tab3}
\end{table}

\begin{table}[htbp]
\caption{The numerical results of the maximum residual error $[maxr_{1n}, maxr_{2n}]$  when $k_{1}=k_{2}=k_{3}=k_{4}=1/2$ of  Example \ref{sec5:eq1}}
\centering
\addtolength{\tabcolsep}{35pt}
\begin{tabular}{c|cccc}
\hline
$n$ & $maxr_{1n}$          & $maxr_{2n}$
\\ \hline
2	&	4.15E-2	&	4.15E-2	\\
3	&	1.41E-2	&	1.41E-2	\\
4	&	5.35E-3	&	5.35E-3	\\
5	&	2.15E-3	&	2.15E-3	\\
6	&	9.05E-4	&	9.05E-4	\\
7	&	3.91E-4	&	3.91E-4	\\
8	&	1.73E-4	&	1.73E-4	\\
9	&	7.83E-5	&	7.83E-5	\\
10	&	3.58E-5	&	3.58E-5	\\
11	&	1.66E-5	&	1.66E-5	\\
\hline
\end{tabular}
\label{tab4}
\end{table}

\subsection{The concentration of the carbon substrate and the concentration of oxygen problem \cite{wazwaz2016variational}}
\begin{example}\label{sec5:eq2}
{\rm Consider the coupled Lane-Emden equations, which was used to study the concentration of the carbon substrate and the concentration of oxygen, as
\begin{align}
\left\{
  \begin{array}{ll}
 \displaystyle  y''_{1}(x)+\frac{\alpha_1}{x}y'_{1}(x)=-b+\frac{a\ y_1(x) \ y_2(x) }{\big(l_{1}+y_1)(m_{1}+y_2\big)}+\frac{c\ y_1(x)\ y_2(x) }{\big(l_{2}+y_1(x))(m_{2}+y_2(x)\big)},~~~~x\in(0,1),\vspace{0.25cm}\\
 \displaystyle  y''_{2}(x)+\frac{\alpha_2}{x}y'_{2}(x)=\frac{d\  y_1(x) \ y_2(x) }{\big(l_{1}+y_1(x)\big) \big(m_{1}+y_2(x)\big)}+ \frac{e\ y_1(x)\ y_2(x) }{\big(l_{2}+y_1(x)\big) \big(m_{2}+y_2(x)\big)},\vspace{0.25cm}\\
y'_{1}(0)=0,~~~y_{1}(1)=1,~~y'_{2}(0)=0,~~~y_{2}(1)=1,
\end{array} \right.
\end{align}
where the parameters $l_1=l_2=m_1=m_2= \nicefrac{1}{10000}, a=5,b=1,c=d= \nicefrac{1}{10}, e=\nicefrac{5}{100}$  are fixed as given in \cite{wazwaz2016variational,duan2015oxygen}. Here,  $a_1=a_2=1,b_1=b_2=0,c_1=c_2=1$.}
\end{example}

In Table \ref{tab5}, \ref{tab6} and \ref{tab7}, we list the numerical results of the approximate solution and the absolute error obtained by the proposed method of  Example \ref{sec5:eq2} for $(\alpha_1=\alpha_2=1$), $\alpha_1=\alpha_2=2$ and $\alpha_1=\alpha_2=3$, respectively. We also compare the numerical results of the maximum residual error $[maxr_{1n}, maxr_{2n}]$ obtained by the present method and the results obtained by the modified ADM \cite{duan2015oxygen} in Table \ref{tab8} for $\alpha_1=\alpha_2=2$.

\subsubsection{When the shape factors $\alpha_1=\alpha_2=1$}
By applying the proposed scheme \eqref{sec2:eq8} with the initial guesses  $y_{10}(x)=1, y_{20}(x)=1$, we obtain the 4-terms series solutions as
\begin{align*}
\left\{
  \begin{array}{ll}
\psi_{14}(x)=2.02484-1.02488 x^2+0.00006968 x^4-0.0000308 x^6+8.57\times 10^{-6} x^8,\\
\psi_{24}(x)=1.0375-0.0374966 x^2+2.049\times 10^{-6} x^4-9.06\times 10^{-7} x^6+2.52\times 10^{-7} x^8.
\end{array} \right.
\end{align*}

\begin{table}[htbp]
\caption{Numerical results of the approximate solution $ [\psi_{1n}(x), \psi_{2n}(x)]$ and the absolute error  $ [r_{1n}(x), r_{2n}(x)]$ when   $\alpha_1=\alpha_2=1$ of  Example \ref{sec5:eq2}}
\centering
\setlength{\tabcolsep}{0.05in}
\begin{tabular}{l|cc| cc|cc|cc}
\hline
\cline{1-9}
$x$ & $\psi_{12}(x)$ & $\psi_{22}(x)$  & $r_{12}(x)$ & $r_{22}(x)$ &$\psi_{14}(x)$ &$\psi_{24}(x)$ &$r_{14}(x)$ &$r_{24}(x)$ \\
\cline{1-9}
0.1	&	2.0145977	&	1.0371205	&	2.61E-4	&	7.68E-6	&	2.0145872	&	1.0371202	&	2.67E-4	&	7.87E-6	\\
0.2	&	1.9838514  &	1.0359956  &	2.49E-4	&	7.33E-6	&	1.9838408	&	1.0359953	&	2.40E-4	&	7.07E-6	\\
0.3	&	1.9326076	&	1.0341208	&	2.29E-4	&	6.76E-6	&	1.9325971	&	1.0341205	&	1.99E-4	&	5.86E-6	\\
0.4	&	1.8608665	&	1.0314960	&	2.03E-4	&	5.98E-6	&	1.8608563	&	1.0314957	&	1.50E-4	&	4.42E-6	\\
0.5	&	1.7686285	&	1.0281214	&	1.70E-4	&	5.01E-6	&	1.7686191	&	1.0281211	&	1.00E-4	&	2.95E-6	\\
0.6	&	1.6558940	&	1.0239968	&	1.32E-4	&	3.90E-6	&	1.6558857	&	1.0239966   &	5.69E-5	&	1.67E-6	\\
0.7	&	1.5226632	&	1.0191224	&	9.16E-5	&	2.69E-6	&	1.5226567	&	1.0191222	&	2.49E-5	&	7.33E-7	\\
0.8	&	1.3689369	&	1.0134981	&	5.07E-5	&	1.49E-6	&	1.3689325	&	1.0134980	&	6.88E-6	&	2.02E-7	\\
0.9	&	1.1947156	&	1.0071239	&	1.61E-5	&	4.76E-7	&	1.1947134	&	1.0071239	&	6.08E-7	&	1.78E-8	\\
\hline
\end{tabular}
\label{tab5}
\end{table}

\subsubsection{For shape factors $\alpha_1=\alpha_2=2$}
Using the proposed scheme \eqref{sec2:eq8} with the initial guesses  $y_{10}(x)=1, y_{20}(x)=1$, we obtain the 4-terms series solutions as
\begin{align*}
\left\{
  \begin{array}{ll}
\psi_{14}(x)=1.66653-0.666545 x^2+0.0000172x^4-5.2743\times 10^{-6} x^6+2.054\times 10^{-6} x^8,\\
\psi_{24}(x)=1.025-0.0249964 x^2+5.172\times 10^{-7} x^4-1.582\times 10^{-7} x^6+6.164\times 10^{-8} x^8.
\end{array} \right.
\end{align*}
\begin{table}[htbp]
\caption{Numerical results of the approximate solution $ [\psi_{1n}(x), \psi_{2n}(x)]$ and the absolute error  $ [r_{1n}(x), r_{2n}(x)]$ when   $\alpha_1=\alpha_2=2$ of  Example \ref{sec5:eq2}}
\centering
\setlength{\tabcolsep}{0.03in}
\begin{tabular}{l|cc |cc|cc|cc}
\hline
\cline{1-9}
$x$ & $\psi_{12}(x)$ & $\psi_{22}(x)$  & $r_{12}(x)$ & $r_{22}(x)$ &$\psi_{14}(x)$ &$\psi_{24}(x)$ &$r_{14}(x)$ &$r_{24}(x)$ \\
\cline{1-9}
0.1	&	1.6598747	&	1.0247462	&	1.31E-4	&	3.94E-6	&	1.6598657	&	1.0247459  &	5.70E-5	&	1.71E-6	\\
0.2	&	1.6398780	&	1.0239963	&	1.25E-4	&	3.75E-6	&	1.6398694	&	1.0239960	&	5.10E-5	&	1.53E-6	\\
0.3	&	1.6065503	&	1.0227465	&	1.14E-4	&	3.44E-6	&	1.6065422	&	1.0227462	&	4.20E-5	&	1.26E-6	\\
0.4	&	1.5598915	&	1.0209967	&	1.00E-4	&	3.01E-6	&	1.5598843	&	1.0209965	&	3.14E-5	&	9.43E-7	\\
0.5	&	1.4999020	&	1.0187470	&	8.34E-5	&	2.50E-6	&	1.4998959	&	1.0187468	&	2.07E-5	&	6.23E-7	\\
0.6	&	1.4265818	&	1.0159974	&	6.38E-5	&	1.91E-6	&	1.4265769	&	1.0159973	&	1.15E-5	&	3.47E-7	\\
0.7	&	1.3399312   &	1.0127479	&	4.31E-5	&	1.29E-6	&	1.3399276	&	1.0127478	&	4.97E-6	&	1.49E-7	\\
0.8	&	1.2399505	&	1.0089985	&	2.32E-5	&	6.97E-7	&	1.2399482	&	1.0089984	&	1.33E-6	&	4.00E-8	\\
0.9	&	1.1266400	&	1.0047492	&	7.12E-6	&	2.13E-7	&	1.1266389	&	1.0047491	&	1.13E-7	&	3.40E-9	\\
\hline
\end{tabular}
\label{tab6}
\end{table}
\subsubsection{For shape factors $\alpha_1=\alpha_2=3$}
Making use of the proposed scheme \eqref{sec2:eq8} with the initial guesses  $y_{10}(x)=1, y_{20}(x)=1$, we obtain the 4-terms series solutions as
\begin{align*}
\left\{
  \begin{array}{ll}
\psi_{14}(x)=1.49989-0.4999 x^2+8.1818\times 10^{-6} x^4-1.295\times 10^{-6} x^6+7.802\times 10^{-7} x^8,\\
\psi_{24}(x)=1.01875-0.018747 x^2+2.454\times 10^{-7} x^4-3.886\times 10^{-8} x^6+2.340\times 10^{-8} x^8.
\end{array} \right.
\end{align*}

\begin{table}[htbp]
\caption{Numerical results of the approximate solution $ [\psi_{1n}(x), \psi_{2n}(x)]$ and the absolute error  $ [r_{1n}(x), r_{2n}(x)]$ when   $\alpha_1=\alpha_2=3$ of  Example \ref{sec5:eq2}}
\centering
\setlength{\tabcolsep}{0.035in}
\begin{tabular}{l|cc |cc|cc|cc}
\hline
\cline{1-9}
$x$ & $\psi_{12}(x)$ & $\psi_{22}(x)$  & $r_{12}(x)$ & $r_{22}(x)$ &$\psi_{14}(x)$ &$\psi_{24}(x)$ &$r_{14}(x)$ &$r_{24}(x)$ \\
\cline{1-9}
0.1	&	1.4948975	&	1.0185594	&	8.20E-5	&	2.46E-6	&	1.4948929	&	1.0185592	&	2.00E-5	&	6.01E-7	\\
0.2	&	1.4799003	&	1.0179970	&	7.79E-5	&	2.33E-6	&	1.4798959	&	1.0179968	&	1.79E-5	&	5.37E-7	\\
0.3	&	1.4549050	&	1.0170596	&	7.12E-5	&	2.13E-6	&	1.4549010	&	1.0170595	&	1.47E-5	&	4.41E-7	\\
0.4	&	1.4199117	&	1.0157473	&	6.21E-5	&	1.86E-6	&	1.4199081	&	1.0157472	&	1.09E-5	&	3.28E-7	\\
0.5	&	1.3749204	&	1.0140601	&	5.11E-5	&	1.53E-6	&	1.3749175	&	1.0140600	&	7.17E-6	&	2.15E-7	\\
0.6	&	1.3199313	&	1.0119979	&	3.88E-5	&	1.16E-6	&	1.3199290	&	1.0119978	&	3.96E-6	&	1.18E-7	\\
0.7	&	1.2549445	&	1.0095608	&	2.59E-5	&	7.77E-7	&	1.2549429	&	1.0095607	&	1.68E-6	&	5.04E-8	\\
0.8	&	1.1799602	&	1.0067488	&	1.37E-5	&	4.12E-7	&	1.1799593	&	1.0067487  &	4.43E-7	&	1.33E-8	\\
0.9	&	1.0949786	&	1.0035618	&	4.12E-6	&	1.23E-7	&	1.0949782	&	1.0035618	&	3.69E-8	&	1.10E-9	\\
\hline
\end{tabular}
\label{tab7}
\end{table}

\begin{table}[htbp]\caption{Comparison of the numerical results of the maximum residual error $[maxr_{1n}, maxr_{2n}]$  when  $\alpha_1=\alpha_2=2$  of  Example \ref{sec5:eq2}}
\centering
\addtolength{\tabcolsep}{17pt}
\begin{tabular}{c|cc|cc}
\hline
\multirow{2}{*} & \multicolumn{2}{c|}{The present method} & \multicolumn{2}{c}{Modified ADM  \cite{duan2015oxygen}}  \\ \cline{2-5}
$n$      & $maxr_{1n}$          & $maxr_{2n}$       & $maxr_{1n}$      & $maxr_{2n}$    
\\ \hline
2	&	1.33E-4	&	4.01E-6	&	1.18E-3	&	3.48E-5	\\
3	&	8.87E-5	&	2.66E-6	&	8.53E-4	&	2.51E-5	\\
4	&	5.91E-5	&	1.77E-6	&	6.21E-4	&	1.82E-5	\\
5	&	3.94E-5	&	1.18E-6	&	4.52E-4	&	1.33E-5	\\
6	&	2.62E-5	&	7.87E-7	&	3.29E-4	&	9.69E-6	\\
7	&	1.74E-5	&	5.25E-7	&	2.39E-4	&	7.06E-6	\\
\hline
\end{tabular}
\label{tab8}
\end{table}

\subsection{The steady-state concentrations of CO$_2$ and PGE \cite{duan2015steady}}
\begin{example}\label{sec5:eq3}
{\rm We consider the  system of nonlinear differential equations, which arising in the study of  the steady-state concentrations
of CO$_2$ and PGE as}
\begin{align}
\left\{
  \begin{array}{ll}
\displaystyle  y''_{1}(x)=\frac{\alpha_{1}y_1(x) \ y_2(x)}{1+\beta_1 y_1(x) +\beta_2 y_2(x)},~~~x\in(0,1),\vspace{0.15cm}\\
\displaystyle  y''_{2}(x)=\frac{\alpha_{2}y_1(x) \ y_2(x)}{1+\beta_1 y_1(x) +\beta_2 y_2(x)},~~~\vspace{0.15cm}\\
y_{1}(0)=1,~~~~y_{1}(1)=k,~~y'_{2}(0)=0,~~~~y_{2}(1)=1,
\end{array} \right.
\end{align}
\end{example}
where the constants $\alpha _1, \alpha _2, \beta _1, \beta _2, k$  are normalized parameters, $x$ is the dimensionless distance
as measured from the center, and $k$ is the dimensionless concentration of CO$_2$  at the
surface of the catalyst \cite{duan2015steady}. We fix the parameters $\alpha _1=1, \alpha _2=2, \beta _1=1, \beta _2=3, k=0.5$ as in  \cite{duan2015steady}. By  applying the proposed scheme, we obtain the approximate series solutions as
\begin{align*}
\left\{
  \begin{array}{ll}
\psi_{14}(x)&=1-0.58005 x+0.09292 x^2-0.01397 x^3+0.00173 x^4-0.000703 x^5+0.000057 x^6\\
          &~~~~~~~~~+0.00001590 x^7+8.1396\times 10^{-7} x^8,\\
\psi_{24}(x)&=0.83989+0.185843 x^2-0.027952 x^3+0.003475 x^4-0.0014068 x^5+0.0001145 x^6\\
&~~~~~~~~+0.0000318 x^7+1.627\times 10^{-6} x^8.
\end{array} \right.
\end{align*}

In Table \ref{tab9}, we list the numerical results of the approximate solution and the absolute error obtained by the proposed method of  Example \ref{sec5:eq3}. We also compare the numerical results of the maximum residual error $[maxr_{1n}, maxr_{2n}]$ obtained by the present method and the results obtained by the numerical method  \cite{hao2018efficient} in Table \ref{tab10}.

\begin{table}[htbp]
\caption{Numerical results of the approximate solution $[\psi_{1n}(x), \psi_{2n}(x)]$ and the absolute error  $[r_{1n}(x), r_{2n}(x)]$ when  $\alpha _1=1, \alpha _2=2, \beta _1=1, \beta _2=3, k=0.5$ of  Example \ref{sec5:eq3}}
\centering
\setlength{\tabcolsep}{0.04in}
\begin{tabular}{l|ll| ll|ll|cc}
\hline
\cline{1-9}
$x$ & $\psi_{12}(x)$ & $\psi_{22}(x)$  & $r_{12}(x)$ & $r_{22}(x)$ &$\psi_{14}(x)$ &$\psi_{24}(x)$ &$r_{14}(x)$ &$r_{24}(x)$ \\
\cline{1-9}
0.1	&	0.9428976	&	0.8415025	&	1.30E-4	&	2.60E-4	&	0.9429113	&	0.8417515	&	8.59E-7	&	1.71E-6	\\
0.2	&	0.8875704	&	0.8468806	&	4.90E-5	&	9.80E-5	&	0.8875994	&	0.8471355	&	3.27E-7	&	6.55E-7	\\
0.3	&	0.8339413	&	0.8556549	&	2.27E-4	&	4.55E-4	&	0.8339853	&	0.8559153	&	1.40E-6	&	2.81E-6	\\
0.4	&	0.7819361	&	0.8676770	&	4.08E-4	&	8.17E-4	&	0.7819931	&	0.8679387	&	2.39E-6	&	4.78E-6	\\
0.5	&	0.7314823	&	0.8828019	&	5.87E-4	&	1.17E-3	&	0.7315485	&	0.8830574	&	3.14E-6	&	6.28E-6	\\
0.6	&	0.6825087	&	0.9008874	&	7.54E-4	&	1.50E-3	&	0.6825785	&	0.9011254	&	3.36E-6	&	6.73E-6	\\
0.7	&	0.6349449	&	0.9217922	&	8.93E-4	&	1.78E-3	&	0.6350110	&	0.9219982	&	2.63E-6	&	5.26E-6	\\
0.8	&	0.5887200	&	0.9453750	&	9.81E-4	&	1.96E-3	&	0.5887737	&	0.9455316   &	4.55E-7	&	9.10E-7	\\
0.9	&	0.5437627	&	0.9714928	&	9.91E-4	&	1.98E-3	&	0.5437943	&	0.9715808	&	3.48E-6	&	6.97E-6	\\
\hline
\end{tabular}
\label{tab9}
\end{table}

\begin{table}[htbp]
\caption{Comparison of the numerical results of the maximum residual errors $[maxr_{1n}, maxr_{2n}]$  when $\alpha _1=1, \alpha _2=2, \beta _1=1, \beta _2=3, k=0.5$ of  Example \ref{sec5:eq3}}
\centering
\addtolength{\tabcolsep}{15pt}
\begin{tabular}{l|ll|ll}
\hline
\multirow{2}{*} & \multicolumn{2}{c|}{The present method} & \multicolumn{2}{c}{Method in \cite{hao2018efficient}}  \\ \cline{2-5}
$n$      & $maxr_{1n}$          & $maxr_{2n}$       & $maxr_{1n}$      & $maxr_{2n}$    \\ \hline
2	&	4.36E-3	&	8.71E-3	&	3.86E-2	&	7.73E-2	\\
3	&	9.99E-4	&	2.00E-3	&	1.14E-3	&	2.28E-3	\\
4	&	3.54E-5	&	7.07E-5	&	4.87E-4	&	9.75E-4	\\
5	&	3.38E-6	&	6.77E-6	&	4.07E-5	&	8.14E-5	\\
6	&	5.29E-7	&	1.06E-6	&	1.01E-6	&	2.03E-6	\\
7	&	1.05E-7	&	2.10E-7	&	3.90E-7	&	7.81E-7	\\
8	&	2.76E-8	&	5.52E-8	&	4.88E-7	&	9.76E-9	\\
9	&	2.80E-9	&	5.61E-9	&	3.64E-9	&	7.28E-9	\\
10	&	2.59E-10&	5.17E-10&	1.22E-10	&	2.43E-10	\\
\hline
\end{tabular}
\label{tab10}
\end{table}

\newpage
\section{Concluding remarks}
The theory of coupled Lane-Emden equation finds its vital presence in many of the natural or physical processes such as occurs in catalytic diffusion reactions \cite{rach2014solving}, and some coupled Lane-Emden equations that relate the concentration of the carbon substrate and the concentration of oxygen \cite{duan2015oxygen}. An analytical approach has been presented for the approximate series solution of the coupled Lane-Emden equation. Unlike the standard Adomian decomposition method, the proposed technique does not require the computation of unknown constants. Unlike the numerical methods,   our approach does not require any linearization or discretization of variables.   Convergence and error estimation of the method is provided.

%\bibliographystyle{elsarticle-num}\small
%\bibliography{randhir_Ham}

\end{document}